\documentclass[oneside,reqno]{amsart}
\usepackage{amssymb,amsmath}
\usepackage{amsfonts,amsthm}
\usepackage{xcolor}
\usepackage{amssymb}
\usepackage{amsfonts}
\usepackage[english]{babel}
\usepackage[utf8]{inputenc}
\usepackage{hyperref}
\usepackage{amsfonts,amsmath,verbatim,amssymb,amscd,latexsym}
\usepackage{amsthm,calc,enumerate, amscd}
\usepackage{graphicx}
\usepackage[T1]{fontenc}
\usepackage{pdfpages}
\usepackage{mathtools} 
\newtheorem{theorem}{Theorem}[section]

\newtheorem{lemma}[theorem]{Lemma}
\newtheorem*{thma}{\bf Theorem A}
\newtheorem*{thmb}{\bf Theorem B}

\newtheorem*{thm1.1}{\bf Theorem 1.1}
\newtheorem*{thm1.2}{\bf Theorem 1.2 (Unconditional)}
\newtheorem*{lem5.3}{\bf Lemma 5.3}

\theoremstyle{definition}

\newtheorem{corollary}{Corollary}

\theoremstyle{remark}
\newtheorem{remark}[]{{ \bf Remark:}}

\numberwithin{equation}{section}
\begin{document}

\title{$t$-aspect subconvexity for $GL(2) \times GL(2)$ $L$-function}
\author{Ratnadeep Acharya, Prahlad Sharma and Saurabh Kumar Singh}
\address{RKMVERI, Belur Math, Howrah-711202, INDIA.}

\email{ratnadeepacharya87@gmail.com}

\address{School of Mathematics, Tata Institute of Funadamental Research, Mumbai-400005} 
\email{prahlad@math.tifr.res.in}
\address{Department of Mathematics and Statistics, Indian Institute of Technology, Kanpur-208016}
\email{skumar.bhu12@gmail.com}

\begin{abstract}
In this paper we shall prove a subconvexity bound for $GL(2) \times GL(2)$ $L$-function in $t$-aspect  by using a $GL(1)$ circle  method. 
\end{abstract}

\maketitle {}

\section{Introduction} 
    One of the interesting problem in analytic number theory is to bound $L$-function on the critical line.  Such a bound may involve one or more parameters. In this article we shall consider the problem for bounding $GL(2) \times GL(2)$ $L$-function in $t$-aspect.  Let $f$ be an Hecke eigenform for the full modular group $SL(2, \mathbb{Z})$ of weight $k$, $g$ be either a Hecke eigenform of weight $k^\prime$ or a weight zero Maass cusp form. The Rankin-Selberg convolution is defined by 
        \begin{align*}
            L(s, f \otimes g) = \zeta(2s) \sum_n \lambda_f(n) \lambda_g(n) n^{-s} \ \ \  (\Re s >1),
        \end{align*} where $ \lambda_f(n)$ and $\lambda_g(n)$ denote the normalised $n$-th Fourier coefficient of $f$ and $g$ respectively.      This extends to a entire function and satisfies a functional equation relating $s$ with $1-s$ (see subsection \ref{rankin section}). The Phragmén-Lindelöf principle gives us the convexity bound $ L(1/2+it, f \otimes g) \ll (|t|+ 10) ^{1+ \epsilon} $.  The Lindel{\" o}f hypothesis predicts that the exponent $1$ can be replace by $0$. 
%The completed $L$ function is defined by 
%\begin{align*}
%    \Lambda(s, f \otimes g) = \frac{1}{(2 \pi)^{2s}} \Gamma \left(s + \frac{k+l}{2}-1 \right)\Gamma \left(s + \frac{k-l}{2} \right)L(s, f \otimes g).
%\end{align*}
% This satisfies the following functional equation: 
% \begin{align*}
%      \Lambda(s, f \otimes g)= \Lambda(1-s, f \otimes g).
% \end{align*}
Our result is the following theorem: 
\begin{theorem}\label{main}
\[L \left(\frac{1}{2}+it, f \otimes g\right) \ll_\epsilon |t|^{1- 1/16 + \epsilon}.\]
\end{theorem} Our method does not depend on the cuspidality $f$ or $g$ (or both). Hence by replacing $g$ to be the Eisenstein series we obtain following corollary :
\begin{corollary}
\[L \left(\frac{1}{2}+it, f\right) \ll_\epsilon |t|^{1/2- 1/32 + \epsilon}.\]
\end{corollary}

We briefly recall the history of $t$-aspect subconvexity bounds. It was first proved by Hardy-Littlewood and Weyl in the case of Riemann zeta function with exponent $1/6$. This was later improved by several eminent mathematicians. In case of degree two $L$-function, subconvexity bound was first obtained by Anton Good \cite{GOOD} by using spectral theory of automorphic forms. X. Li \cite{li} proved subconvexity bound for self-dual degree three $L$-function by using moment method, where the positivity of central value was essential. For general degree three $L$-function, subconvexity bound was obtain by Ritabrata Munshi \cite{RM} by using his conductor lowering trick.  K. Aggarwal and S. K. Singh \cite{RM} adopted conductor lowering trick for $GL(2)$ $L$-functions, and they proved Weyl bound for the same. In their fundamental work, P. Michel and A. Venkatesh \cite{venkatesh} have proved a subconvexity bound for general $GL(2)$ and $GL(2) \times GL(2)$ $L$-functions. They proved following two theorem:
\begin{thma}
There is an absolute constant  $\delta >0$ such that: for an automorphic repre-
sentation of $GL_1(\mathbb{A}) $ or $GL_2(\mathbb{A}) $ (with unitary central character), one has
\begin{equation*}
L (1/2, \pi )\ll_{\mathbb{F} }  C(\pi)^{1/4 - \delta}.
\end{equation*}
\end{thma}

\begin{thmb}
There is an absolute constant $\delta >0$ such that: for $\pi_1, \pi_2$ automorphic representations on $GL_2(\mathbb{A}_{\mathbb{F} } )$ we have
\begin{equation*}
L (1/2, \pi_1 \otimes \pi_2 )\ll_{\mathbb{F} }  C(\pi)^{1/4 - \delta};
\end{equation*} more precisely, the implied constant  depends polynomially on the discriminant of $\mathbb{F} $ ( for  $\mathbb{F} $  varying over fields of given degree) and on $C(\pi_2)$.
\end{thmb}
As a consequence of Theorem $B$, we obtain that $L(1/2+ it, f \otimes g) \ll |t|^{1 - \delta}$ for some positive $\delta$. In our theorem we shall prove that $\delta$ can be taken $0\leq  \delta < 1/16$. 

% In the case of degree three $L$-function, it was proved by X. Li for 

\section{The Set up}
   By means of approximate functional equation and dyadic subdivision, we have 
   \begin{align}\label{afe}
     L \left(\frac{1}{2}+it, f \otimes g \right) \ll_{\epsilon, A} \sup_{N \ll t^{2+ \epsilon}} \frac{|S(N)|}{\sqrt{N}}+ t^{-A}  
   \end{align}
   for any small $\epsilon >0$ and any large $A>0$, where
       \[S(N)= \sum_{n \sim N} \lambda_f(n) \lambda_g(n) n^{-it}.\]
 \subsection{The delta method}
We separate the oscillations from $\lambda_{f}(n)$ and $\lambda_{g}(n)n^{-it}$ using a version of the delta method due to Duke, Friedlander and Iwaniec. More specifically we will use the expansion $(20.157)$ given in Chapter 20 of \cite{iwaniec}. Let $\delta : \mathbb{Z}\to \{0,1\}$ be defined by
\[
\delta(n)=
\begin{cases}
1 &\text{if}\,\,n=0 \\
0 &\text{otherwise}
\end{cases}
\]
Then for $n\in\mathbb{Z}\cap [-2M,2M]$, we have
\begin{equation}\label{s}
\delta(n)=\frac{1}{Q}\sum_{a\bmod q}e\left(\frac{na}{q}\right)\int_{\mathbb{R}}g(q,x)e\left(\frac{nx}{q	Q}\right) dx
\end{equation} where $Q=2M^{1/2}$. The function $g$ satisfies the following property (see $(20.158)$ and $(20.159)$ of \cite{iwaniec}).
\begin{equation}\label{delta}
\begin{aligned}
&g(q,x)=1+h(q,x),\,\,\,\,\text{with}\,\,\,\,h(q,x)=O\left(\frac{1}{qQ}\left(\frac{q}{Q}+|x|\right)^A\right)\\
&g(q,x)\ll |x|^{-A}
\end{aligned}
\end{equation}for any $A>1$. In particular the second property imples that the effective range of integral in \eqref{s} is $[-M^{\epsilon},M^{\epsilon}]$.
\subsection{Conductor lowering}
We write $S(N)$ as
\begin{equation}
S(N)=\frac{1}{K}\sum_{n\sim N}\sum_{m\sim N}\lambda_f(n) \lambda_g(m) m^{-it}\delta(n-m)\int_{\mathbb{R}} V\left(\frac{\nu}{K}\right) \left(\frac{m}{n} \right)^{i \nu} \ d\nu.
\end{equation} Note that the $\nu$ integral is negligibly small unless $m-n\ll N/K $. Now using \eqref{s} with $Q=\sqrt{N/K}$ we get
\begin{align}
     S(N)=& \frac{1}{KQ} \sum_{n \sim N} \sum_{m \sim N}  \lambda_f(n) \lambda_g(m) m^{-it} \sum_{q \leq Q}\frac{1}{q} \sum_{a(q)}^* e \left( \frac{a(n-m)}{q}\right) \times \nonumber\\
&  \int_{\mathbb{R}} g(u,q) e \left( \frac{(n-m)u}{qQ}\right) \ du \int_{\mathbb{R}} V\left(\frac{\nu}{K}\right) \left(\frac{m}{n} \right)^{i \nu} \ d\nu+ O_{A}(t^{-A}).
\end{align}
Next we divide the range of the sum over $q$ into dyadic intervals as follows :
\begin{equation}\label{dyadic}
    S(N)=\sum_{C}S_C(N)+O_A(t^{-A})
\end{equation}with
\begin{equation*}
    \sum_{C}1\ll \log Q\ll t^{\epsilon},
\end{equation*}where
\begin{align}
     S_C(N):&= \frac{1}{KQ} \sum_{n \sim N} \sum_{m \sim N}  \lambda_f(n) \lambda_g(m) m^{-it} \sum_{q \sim C}\frac{1}{q} \sum_{a(q)}^* e \left( \frac{a(n-m)}{q}\right) \times \nonumber\\
&  \int_{\mathbb{R}} g(u,q) e \left( \frac{(n-m)u}{qQ}\right) \ du \int_{\mathbb{R}} V\left(\frac{\nu}{K}\right) \left(\frac{m}{n} \right)^{i \nu} \ d\nu.
\end{align}
\noindent
By introducing smooth partition of unity we rewrite $S_C(N)$ as 
\begin{align}\label{S_*(N)}
   S_C(N):&= \frac{1}{KQ}\sum_{q \leq Q}\frac{1}{q} \sum_{a(q)}^*   
  \int_{\mathbb{R}} g(u,q)  \ dx \int_{\mathbb{R}} V\left(\frac{\nu}{K}\right) \left(\frac{m}{n} \right)^{i \nu} \ d\nu \ \times \nonumber\\ &\left[\sum_m \lambda_g(m) e\left( \frac{-am)}{q}\right)m^{-i(t- \nu)}e \left( \frac{-mu}{qQ}\right) \omega\left( \frac{m}{N}\right)\right] \nonumber\\ &\left[\sum_n \lambda_f(n) e\left( \frac{an)}{q}\right)n^{-i\nu}e \left( \frac{nu}{qQ}\right) \omega\left( \frac{n}{N}\right)\right].
\end{align}
\section{Sketch of the proof} In this section we shall give a sketch of the proof. We shall consider the generic case $m, n \sim N$ $ N \sim t^2$ and $q \sim Q$.  Introduction of delta symbol gives us a loss of size $N$. To obtain a sub-convexity bound, we need to save $N$ and little more. We reach our goal in following steps.

{\bf Step 1: First application of Voronoi summation formula} We first apply the Voronoi summation formula to $m$ sum given in equation \eqref{S_*(N)}.  Its initial length is of size $N$ and the conductor is of size $q^2 t^2$ (as we choose $K \ll t^{1- \delta}$ for some $\delta>0$.) We obtain that dual length is essentially supported on a sum of size $\ll M_0= q^2 t^2/ N+K .$  In this step we obtain a saving a size $N/ (q t)$. 

{\bf Step 2: Second application of Voronoi summation formula} We then apply the Voronoi summation formula second time to $n$ sum given in equation \eqref{S_*(N)}. Its Initial length is of size $N$ and conductor is of size $q^2 k^2$. We note that dual length is essentially supported on a sum of size $\ll N_0= q^2 K^2/ N+K .$  In this step we obtain a saving a size $N/ (q K)$. After the summation formulae, we obtain the following expression for $S(N)$

\begin{align*} 
    \int_{\mathbb{R}} V(\nu) \ d\nu \sum_{q \sim Q} \int_{\mathbb{R}} g(u,q) \ du \sum_{m \sim M_0} \sum_{n \sim N_0} \sideset{}{^\star } \sum_{a(q)} e \left( \frac{a (m-n)}{q}\right) \lambda_f(n)  \lambda_g(m) I(n, q, \nu) I(m, q, t- \nu)
\end{align*}  where $ I(n, q, \nu)$ is given by equation \eqref{first integral}.  

{\bf Step 3: Sum over $a$ and evaluation of integrals} Summation over $a$ is a Ramanujan sum,  evaluation of which gives us a saving of size $Q$ at the cost of a congruence condition (usually a character sum gives a square-root cancellation). We are also able to save $\sqrt{K}$ from $v$ integral. Total saving after the third step is given by:
\begin{align*}
    \frac{N}{Qt} \times \frac{N}{QK} \times Q \times \sqrt{K} = \frac{N^2}{ t  Q\sqrt{K}} = N
\end{align*}
We are on the boundary and we need to save little more. Let $t^{\delta/2}$ be the desired saving. We have a following expression for $S(N)$
\begin{align*}\label{supremum}
    S_{0,C}(N) \sim \sum_{q \sim Q} \sum_{d \mid q} d \sum_{m \sim M_0} \left|\lambda_g(m)\right| \left|\mathop{\sum\sum}_{\substack{ \ n \leq N_0\\m \equiv n (d)}} \lambda_f(n)   \mathcal{I}_{y, \nu}(m,n;q)\right|,
\end{align*} where $ \mathcal{I}_{y, \nu}(m,n;q)$ is given by equation \eqref{second last integral}. We also have that oscillation of $m$ in integral $ \mathcal{I}_{y, \nu}(m,n;q)$ is of size $t$ (see equation \eqref{oscillation of m}).

{\bf Step 4: Cauchy inequality and Poisson summation formula} We now apply C-S inequality followed by Poisson summation formula to the $m$ sum. Note that $N_0 \sim K$ and $M_0 \sim t^2/K$. The ``analytic conductor" is of size $t$ and the ``arithmetic conductor" coming from congruence is of size $Q$, which gives a dual length of size $Qt/ M_0$.  Diagonal is of size $K/d \gg K/Q$ and we are able to save whole diagonal in the case of zero frequency. Saving for diagonal terms is enough if 
\begin{align*}
    \frac{K}{d} \gg t^\delta \Leftrightarrow K \gg t^\delta \frac{t}{\sqrt{K}} \Leftrightarrow K > t^{ (1+ \delta) \frac{2}{3} }. 
\end{align*} From the non-diagonal part, we save $M_0/ Q\sqrt{ t} \sim t^{1/2}/ K^{1/2}$. Hence the saving from off-diagonal terms is enough if 
\begin{align*}
    \frac{t^{1/2}}{K^{1/2}} \gg t^\delta \Leftrightarrow K \ll t^{1- 2 \delta}. 
\end{align*} To obtain our sub-convexity result, we choose $K$ such that $ t^{ (1+ \delta) \frac{2}{3} } < K <t^{1- 2 \delta} $. We choose $K$ so that the savings from the diagonal term and non-diagoanl term are of same order, that is,  
\begin{align*}
    \frac{K}{Q} = \frac{t^{1/2}}{K^{1/2}} \Leftrightarrow  K^2 = t^{3/2} \Leftrightarrow K= t^{3/4}
\end{align*}

%  We choose $K$ so that the savings from the diagonal term and non-diagoanl term are of same order, that is,  
%\begin{align*}
%    \frac{K}{Q} = \frac{t^{3/2}}{K^{3/2}} \Leftrightarrow  K^3 = t^{5/2} \Leftrightarrow K= t^{5/6}
%\end{align*}

\section{Preliminaries}
In this section, we shall recall some basic facts about $SL(2, \mathbb{Z})$ automorphic forms and newforms (for details see \cite{HI} and \cite{IK}).  Detail exposition  for Hecke operators, pseudo eigenvalue and Rankin-Selberg convolution is given in \cite{AL}, \cite{LI1}, \cite{LI2} and \cite{KMV}. 
\subsection{Holomorphic cusp forms} 

Let $\chi_q$ be a Dirichlet character of modulus $q$, conductor $q_0$ and let $k \geq 2$ be an integer. Let $S(q, \chi_q ; k)$ denote the vector space of modular forms of weight $k$ level $q$ and character $\chi_q$. For $f \in S(q, \chi_q ; k)$, we represent $f$  by  Fourier expansion 
$$ f(z)= \sum_{n=1}^\infty \psi_f(n) n^{(k-1)/2} e(nz),$$
where $ e(z) = e^{2\pi i z}$ and $\lambda_f(n), \ {n \in \mathbb{Z}}$ are the  Fourier coefficients.   This space is equipped with the Peterson inner product
\begin{equation}
(f, g)_k := \int_{ \Gamma_0 (q) \setminus \mathbb{H}} f(z) \overline{g(z)} \, \, ( \Im z)^k  d\mu z. 
\end{equation}
For $(n, q)= 1$ let $T_n$ denote  the Hecke operators with Hecke eigenvalue $\lambda_f (n) n^{(k-1)/2}$. Let  $T_n^\star = \overline{\chi(n)}  T_n$ denotes the adjoint of $T_n$. Let $\mathbb{B}(q, \chi_q ; k) $  denotes an orthonormal basis  of $S(q, \chi_q ; k)$, which consists of eigenvectors of all the $T_n$, $(n, q)=1$.  Let $S(q, \chi_q ; k)^{\textrm{old}}$ denote the subspace $S(q, \chi_q ; k)$ generated by  the forms $ f(dz)$ with $f \in S(q^\prime, \chi_{q^\prime} ; k)$, $dq^\prime \mid q$  $q^\prime \neq q$, and $\chi_{q^\prime} $ inducing  $\chi_q$.  Let $S(q, \chi_q ; k)^{\textrm{new}}$ denotes the orthogonal compliments of $S(q, \chi_q ; k)^{\textrm{old}}$  with respect to the Peterson inner product. This space can be simultaneously diagonalised and for such $f$ we have  $ \psi_f(n) \neq 0$. We say that $f \in S(q, \chi_q ; k)^{\textrm{new}}$ is normalised if $ \psi_f(n) =1$ and in this case $ \psi_f(n) = \lambda_f (n)$.  We denote $S(q, \chi_q ; k)^\star$ the normalised orthonormal basis of $S(q, \chi_q ; k)^{\textrm{new}}$.  Deligne proved that $|\lambda_f(n)| \leq d(n)$, where $d(n)$ is the divisor function. $L$-function associated with the form $f$ is given by 

\begin{align} \label{gl2 l function}
L( s, f )= \sum_{n=1}^\infty \frac{\lambda_f(n)}{n^s} \ &=  \prod_p \left( 1 -\lambda_f(p) p^{-s} + \chi_q(p) p^{-2s} \right)^{-1} \notag \\ 
& =  \prod_p \left( 1 - \frac{\alpha_{f, 1}(p)}{ p^s}\right)^{-1}  \left( 1 - \frac{ \alpha_{f, 2}(p)}{ p^s}\right)^{-1},   
\end{align} where $(\Re s>1).$ The completed $L$-function is given by 
\[
\Lambda(s, f) : = ( 2 \pi)^{-s} \Gamma \left( s + \frac{k-1}{2}\right) L( s, f ) = \pi^{-s} \Gamma\left( \frac{s + (k+1)/2}{2}\right)  \Gamma\left( \frac{ s + (k-1)/2}{2}\right)L( s, f ). 
\]

 Hecke proved that $L(s, f)$ admits an analytic continuation to the whole complex plane and satisfies the functional equation
\begin{align*} 
 \Lambda(s, f) = \epsilon(f) \  \frac{\tau (\chi)}{\sqrt{q}} \Lambda(1-s,\overline{f}),
\end{align*}
 where $ \epsilon(f)$ is  a root number and $\overline{f} $ is the dual form of $f$.

\subsection{Maass cusp forms} Let $\chi_D$ be a Dirichlet character of modulus $D$, conductor $D_0$ and let $\Delta$ denotes the hyperbolic Laplacian. Let $\lambda = \frac{1}{4} + \nu^2$ denotes an eigenvalue of $\Delta$, that is, there exists an function on upper half plan such that $(\Delta + \lambda) f = 0 $. Let $M(D, \chi_D; \lambda)$ denote the finite dimensional vector space of weight zero Maass forms of level $D$, nebentypus $\chi_D$, and eigenvalue  $\lambda$. Let $S ( D, \chi_D; \lambda)^\star$ denotes the subspace primitive Maass cusp form. We represent $f$ by Fourier expansion
\[
f(z)= \sqrt{y} \sum_{n \neq 0} \lambda_f(n) K_{ i \nu} (2 \pi |n|y) e(nx), 
\] 
where $ K_{ i \nu}(y)$ is the  Bessel function of  second kind. Ramanujan-Petersson conjecture predicts that $|\lambda_f(n)|\ll n^\epsilon$.  The work of H. Kim and P. Sarnak \cite{KS} tells us  that  $|\lambda_f(n)|\ll n^{7/64+\epsilon}$. $L$-function associated with the form $f$ is defined by 
\[
L( s, f )= \sum_{n=1}^\infty \frac{\lambda_f(n)}{n^s} \  =   \prod_p \left( 1 - \frac{\alpha_{f, 1}(p)}{ p^s}\right)^{-1}  \left( 1 - \frac{ \alpha_{f, 2}(p)}{ p^s}\right)^{-1},  \ \ \ \   \Re \ s>1.
\] 
 It extends to an entire function and satisfies the functional equation 
$ \Lambda(s, f) = \epsilon(f ) \Lambda(1-s, \overline{f})$, where $ |\epsilon(f )| = 1$   and completed $L$-function  $ \Lambda(s, f)$ is given by
\[
\Lambda(s, f) = \pi^{-s} \Gamma \left( \frac{s  + i \nu   }{ 2}  \right)   \Gamma \left( \frac{s  - i \nu }{ 2} \right) L(s, f) . 
\]  

\subsection{Review of Rankin-Selberg convolution} \label{rankin section}
 Let $\chi_q$ and $\chi_D$ be Dirichlet character of modulus $q$ and $D$ respectively.  Let $\chi:= \chi_q \chi_D$ denote the Dirichlet character modulo l.c.m. of $q$ and $D$.  Let $ f\in S(q, \chi_q ; k)^\star$ and $g \in S(q, \chi_q ; \star)^\star$, where $ \star$ is either equal to $k^\prime$ or equal to $ \lambda = 1/4 + r^4$ with $\lambda \geqslant 1/4$. Rankin-Selberg convolution of $f$ and $g$ is defined by 
\begin{align}
L(s, f\otimes g)= L(2s, \chi) \sum_{n=1}^\infty \frac{\lambda_f(n) \lambda_g(n)}{ n^s}= \prod_p \prod_{i = 1}^2   \prod_{j = 1}^2 \left( 1 - \frac{\alpha_{f, i} (p)\, \alpha_{g, j} (p)}{p^s}\right)^{-1}.
\end{align} where $ \alpha_{f, i}$ is as given in equation \eqref{gl2 l function}.  Rankin- Selberg prove that $L(s, f\otimes g) $ admits analytic continuation to whole complex plane except when $g = \overline{f}$ in which case it has a simple pole at $s=0, 1$. The completed $L$-function is given by 
\begin{align*}
\Lambda(s, f\otimes g) =  \left( \frac{q D}{4 \pi^2}\right)^s  \Gamma_{f, g} (s) L(s, f\otimes g), \, \, \, \, \textrm{where}
\end{align*} 
\begin{align*}
\Gamma_{f, g} (s) = 
\begin{cases} 
\Gamma \left( s + \frac{|k - k^\prime|}{2}\right)  \Gamma \left( s + \frac{k +k^\prime}{2}- 1\right)  \ \ \ \ \textrm{for} \  g  \  \textrm{is holomorphic of weight } k^\prime \\ 
\Gamma \left( s + \frac{k + 2 ir-1}{2}\right)  \Gamma \left( s + \frac{k -2 ir-1}{2}\right)  \ \ \ \ \textrm{otherwise}. 
\end{cases}
\end{align*} $\Lambda(s, f\otimes g)$ satisfy a functional equation
\begin{align*}
\Lambda(s, f\otimes g)  = \varepsilon (f\otimes g) \, \, \Lambda(1-s, \overline{f} \otimes  \overline{g})  \, \, \, \, \textrm{where}, 
\end{align*}
\begin{align*}
 \varepsilon (f\otimes g) = 
\begin{cases} 
\chi_D(-q) \chi_q (D) \, \,  \eta_f (q)^2 \, \,  \eta_g (D)^2    \ \ \ \ \textrm{if } \  g  \  \textrm{is holomorphic} k^\prime  \geqslant k\\ 
\chi_D(q) \chi_q (-D) \, \,  \eta_f (q)^2 \, \,  \eta_g (D)^2  \ \ \ \ \textrm{for} \  g  \  \textrm{Maass form},
\end{cases}
\end{align*} where $\eta_g(q)$ is pseudo eigenvalue of Atkin-Lehner-Li operator $W_q$, and $g \mid W_q = \eta_g(q) \overline{g}$. We now recall the Voronoi summation formula for $GL(2)$  automorphic form.
\begin{lemma} \label{gl2 voronoi}
 Let $f$ be a Hecke cusp form on modular group $SL(2, \mathbb{Z}) $ with normalised  Fourier coefficients $\lambda_f(n)$. Let $a, c$ be integers such that $(a, c)=1$.   Let $V$ be a smooth compactly supported function on $\mathbb{R}$.   We have
 \begin{equation}
 \resizebox{\textwidth}{!}
     {
       $ \sum_{n=1}^\infty \lambda_f(n) e\left( \frac{n a}{q}\right) V \left( \frac{n }{X}\right) = \mathcal{V}(f; q, X) + \frac{X}{q} \sum_{\pm} \sum_{n=1}^\infty \lambda_f(\pm n)  e\left( \frac{n a}{q}\right) e\left( \frac{\mp n \overline{a} n}{q}\right) F_{\pm} \left( \frac{n X}{q^2}\right)
                  \textsubscript{}  $
     }
% 
% \sum_{n=1}^\infty \lambda_f(n) e\left( \frac{n a}{q}\right) V \left( \frac{n }{X}\right) = \mathcal{V}(f; q, X) + \frac{X}{q} \sum_{\pm} \sum_{n=1}^\infty \lambda_f(\pm n)  e\left( \frac{n a}{q}\right) e\left( \frac{\mp n \overline{a} n}{q}\right) F_{\pm} \left( \frac{n X}{q^2}\right), 
\end{equation}  
where 
\begin{align*}
\mathcal{V}(f; q, X) = 
\begin{cases} 
\frac{X}{q} \int_{\mathbb{R}} V(x) (\log xN + \gamma - 2\log \, q) {d} x \    \textrm{if } \lambda_f(n) \textrm{is the divisor function } \tau(n) \\
0  \hspace{7cm} \textrm{otherwise}. 
\end{cases}
\end{align*}
Here $\gamma$ is the Euler constant.  $F_{\pm} $  and $F_{-} $ are integral transform of $V$ given by the following:
\begin{itemize}
\item If $f$ is an holomorphic form of weight $k$ the 
$F_{+} $ is given by 
\begin{equation*}
F_+= 2 \pi i^k \int_0^\infty V (y) J_{k-1} \left( 4 \pi \sqrt{Xy}\right) dy, \ \ \ \ \ \textrm{and} \ \ \ F_- = 0. 
\end{equation*} 

\item If $f$ is a Maass form with Laplacian eigenvalue $1/4+ r^2$ and let $\\varepsilon_f$ the eigenvalue of involution operator then 
\begin{align*}
& F^+(x) = \frac{- \pi}{ \sin  i \pi r} \int_0^\infty V(y) \left( J_{2 i  r} (\sqrt{4 \pi xy} ) - J_{-2 i  r} (4 \pi \sqrt{xy} ) \right) \  dy \ \ \ \ \ \textrm{and } \\ 
& F^-(x) = 4 \varepsilon_f \cosh \pi r  \int_0^\infty V(y) K_{2 i  r} (\sqrt{4 \pi xy} ) \  {d}y. 
 \end{align*} 
 If $r=0$ then 
 \begin{align*}
F^+(x) = - 2 \pi  \int_0^\infty V(y) Y_0 (\sqrt{4 \pi xy} )  \  dy, \  \textrm{and} \ \ F^-(x) = 4 \varepsilon_f  \int_0^\infty V(y) K_{0} (\sqrt{4 \pi xy} ) \  {d}y. 
\end{align*}
\item If $\lambda_f(n)= \tau(n)$, where $\tau(n)$ denotes the number of divisors of $n$, then
\begin{align*}
F^+(x) = - 2 \pi  \int_0^\infty V(y) Y_0 (\sqrt{4 \pi xy} )  \  dy, \  \textrm{and} \ \ F^-(x) = 4  \int_0^\infty V(y) K_{0} (\sqrt{4 \pi xy} ) \  {d}y. 
\end{align*} 

\item  In each case, we have 
\begin{equation}
F^{\pm} (y) \ll_{A} (1+ |y|)^{-A}, 
\end{equation} for any $A\geq 0$. 

\end{itemize}

 \end{lemma}

 \begin{proof}
 See \cite[page $185$]{KMV}.
 \end{proof}

\section{Some Lemmas}

In this section we shall recall some results which we  require in the sequel. We first recall   Rankin-Selberg bound for the Fourier coefficients in the following lemma. 

\begin{lemma} \label{rankin Selberg bound}
Let $\lambda_f(n)$ be the Fourier coefficients of a holomorphic cusp form, or a Maass form. For any real number $x\geq 1$, we have 
\begin{align*}
\sum_{1\leq n \leq x} \left| \lambda_f(n) \right|^2 \ll_{f, \epsilon} x^{1+\epsilon}. 
\end{align*} 

\end{lemma}

We also require to estimate the exponential integral of the form: 
\begin{equation} \label{eintegral}
\mathfrak{I}= \int_a^b g(x) e(f(x)) dx,
\end{equation} where $f$ and $g$ are  real valued smooth functions on the interval $[a, b]$. We recall the following lemma on exponential integrals.

\begin{lemma} \label{sdb}
Let $f$ and $g$ be real valued twice differentiable function and let $f^{\prime \prime} \geq r>0$ or  $f^{\prime \prime} \leq -r <0$, throughout the interval $[a, b]$. Let $g(x)/f^\prime(x)$ is monotonic and $|g(x)| \leq M$. Then we have

\begin{align*}
\mathfrak{I} \leq \frac{8M}{\sqrt{r}}. 
\end{align*} 
\end{lemma} 
\begin{proof}
See \cite[Lemma 4.5, page 72]{ECT}
\end{proof}

\begin{lemma} \label{poisson}
 {\bf Poisson summation formula}: $f:\mathbb{R } \rightarrow \mathbb{R}$ is any Schwarz class function. Fourier transform 
 of $f$ is defined  as 
 \[
  \widehat{f}(y) = \int_{ \mathbb{R}} f( x) e(- x   y) dx,
 \] where $dx$ is the usual Lebesgue measure on $ \mathbb{R } $. 
 We have 
\begin{equation*}
 \sum_{ n \in \mathbb{Z}  }f(n) = \sum_{m \in \mathbb{Z} } \widehat{f}(m). 
\end{equation*} If $W(x)$ is any smooth and compactly supported function on $\mathbb{R}$, we have:
\begin{align*}
\sum_{n \in \mathbb{Z}  }e\left( \frac{an}{q}\right) W\left( \frac{n}{X}\right) = \frac{X}{q} \sum_{ m \in \mathbb{Z}  } \sum_{\alpha (\textrm{mod} \ q )}  e\left(\frac{\alpha + m}{q} \right) \widehat{W} \left( \frac{mX}{q} \right). 
\end{align*} 
\end{lemma}
 \begin{proof}
 See  \cite[page 69]{IK}.
 \end{proof}

\begin{remark} \label{remark}
If $W(x)$ satisfies $x^j W^{{j}} (x)\ll1$, then it can be easily shown, by integrating by parts  that dual sum is essentially supported on $ m\ll \frac{q (qX)^{\epsilon}}{X}$. The contribution coming from $m\gg \frac{q (qX)^{\epsilon}}{X} $ is negligibly small. 
\end{remark}

\section{Proof of the Theorem}
\subsection{Application of Voronoi summation formula}
Here we evaluate the sums over $m,n$ which appeared in \eqref{S_*(N)}. The Voronoi summation formula transforms the $m$-sum essentially to 
\[ \frac{N^{\frac{3}{4}-i(t- \nu)}}{\sqrt{q}} \sum_{m \geq 1}\frac{\lambda_g(m)}{m^{1/4}} e \left(\frac{\overline{a}m}{q} \right) \int_0^ \infty \omega(x)x^{-i(t- \nu)} e\left(-\frac{Nux}{qQ} \pm \frac{2 \sqrt{mNx}}{q}  \right) \ dx.\]
Similarly, the evaluation of the $n$-sum is essentially given by
\[ \frac{N^{\frac{3}{4}-i\nu}}{\sqrt{q}} \sum_{n \geq 1}\frac{\lambda_f(n)}{n^{1/4}} e \left(\frac{-\overline{a}n}{q} \right) \int_0^ \infty \omega(y)y^{-i \nu} e\left(\frac{Nuy}{qQ} \pm \frac{2 \sqrt{nNy}}{q}  \right) \ dy.\]
Moreover, by repeated integrating by parts we can truncate the aforementioned dual sums in $m, n$ upto  negligible error terms. Precisely, the lengths of the dual variables reduces to $m \ll M_0:= (qt)^2/N+K$ and $n \ll N_0:=(qK)^2/N+K$.
 \noindent
 Thus we obtain four identical sums by putting the above evaluations of the sums over $m$ and $n$ in \eqref{S_*(N)}. Without loss of generality we will consider only a representative $S_{0,C}(N)$ which is given by

\begin{align}\label{S(N)}
    \frac{N^{\frac{3}{2}-it}}{Q} \int_{\mathbb{R}} V(\nu) \ d\nu \sum_{q \sim C } \frac{1}{q^2} \int_{\mathbb{R}} g(u,q) \ du \sum_{m \leq M_0} \sum_{n \leq N_0} \frac{\lambda_f(n)  \lambda_g(m)}{(mn)^{1/4}} R_q(m-n) I(n, q, \nu) I(m, q, t- \nu),
\end{align}
where
\begin{equation} \label{first integral}
    I(n, q, \nu)= \int_0^\infty \omega(x) x^{-i\nu} e\left(\frac{Nux}{qQ} + \frac{2 \sqrt{nNx}}{q}  \right) \ dx.
\end{equation}
We recall that the Ramanujan sum satisfy the following identity:
\[ R_q(m):= \sum_{a (q)}^* e \left(\frac{am}{q} \right) = \sum_{d \mid (q, m-n)}d \mu\left(\frac{q}{d} \right).\] This yields 
\begin{align}\label{S_C(N)}
   &S_{0,C}(N)= \frac{N^{\frac{3}{2}-it}}{Q} \int_{\mathbb{R}} V(\nu) \ d\nu \sum_{q \sim C } \frac{1}{q^2} \sum_{d \mid q} d \mu(q/d) \int_{\mathbb{R}} g(u,q) \ du \times \nonumber \\&\mathop{\sum\sum}_{\substack{m \leq M_0 \ n \leq N_0\\m \equiv n (d)}} \frac{\lambda_f(n)  \lambda_g(m)}{(mn)^{1/4}}  I(n, q, \nu) I(m, q, t- \nu),
\end{align}
\subsection{Simplification of the integrals}
After the application of summation formulae, we end up with a $4$ fold integral which is given by 
\begin{align} \label{all 4 integrals}
    \mathfrak{J} (m, n; q) = \int_{\mathbb{R}} V(\nu) \ d\nu \int_0^ \infty \omega(x)x^{-i(t- \nu)} & e\left( \frac{2 \sqrt{mNx}}{q}  \right) \ dx \int_0^ \infty \omega(y)y^{-i \nu} e\left(\frac{2 \sqrt{nNy}}{q}  \right) \ dy \notag\\
    & \times \int_{\mathbb{R}} g(u,q) e\left(\frac{Nu(y-x)}{qQ}  \right)\ du
\end{align}. We consider the $u$-integral.  From the splitting of $g(u, q)= 1+ h(u, q)$ (see \eqref{delta}) we observe that integral over $u$ splits into two integral of the form
\begin{align*}
    \int_{\mathbb{R}}  e\left(\frac{Nu(y-x)}{qQ}  \right)\ du + \int_{\mathbb{R}} h(u,q) e\left(\frac{Nu(y-x)}{qQ}  \right)\ du, 
\end{align*} where  weight function $ h(u,q)$ is of smaller order (see equation \eqref{delta}). Integrating by parts we observe that first integral is negligibly small unless
\begin{align*}
    |y-x| \ll t^\epsilon \frac{q Q}{N} \ll  t^\epsilon \frac{q }{Q K}. 
\end{align*} For the estimation of second integral we consider $\nu $ integral in equation \eqref{all 4 integrals} and obtain that $ \mathfrak{J} (m, n; q)$ is negligibly small unless $|y-x| \ll   \frac{t^\epsilon}{ K}$. We also have a saving of size $q Q$ due to the size of weight function $ h(u,q)$. As a result we obtain much stronger bound in this case. From now on we proceed with the analysis of integral by considering first integral. Writing $y -x= y_1$ with $y_1 \ll \frac{q}{Q K}$, the integral $ \mathfrak{J} (m, n; q)$ essentially reduce to  $ \mathcal{I}_{y, \nu}(m,n;q)$ where

\begin{equation*} 
    \mathcal{I}_{y_1, \nu}(m,n;q)=\int_0^ \infty \omega_{y_1,\nu}(x)x^{-it}e\left( \frac{2 \sqrt{mNx}}{q} + \frac{2 \sqrt{nN(x+y_1)}}{q} \right) \ dx \ \ \textnormal{and}
\end{equation*}
\[\omega_{y_1,\nu}(x)= \omega(x) \omega(x+y_1)\left(1+\frac{y_1}{x}\right)^{-i\nu} \ dx. \]

Since $y_1 \ll t^\epsilon q/QK \ll t^\epsilon /K $ we observe that 
\begin{align*}
    \frac{\partial^j }{\partial x^j} \omega_{y_1,\nu}(x) \ll_j \left( 1 + y_1 \nu\right)^j \ll_j t^{\epsilon\, j}
\end{align*} Above calculation shows that the oscillation of $\omega_{y_1,\nu}(x)$ with respect to $x$ is of size  $t^\epsilon$. Hence we can absorb $\left(1+\frac{y_1}{x}\right)^{-i\nu}$ into the weight function, and call the new smooth and compactly supported weight function by $\omega (x)$ with understanding that that  $\omega^{(j)} (x) \ll_j t^{\epsilon\, j}$. We obtain
\begin{equation} \label{second last integral}
    \mathcal{I}_{y_1, \nu}(m,n;q) := \mathcal{I} (m,n;q)=\int_0^ \infty \omega(x) x^{-it}e\left( \frac{2 \sqrt{mNx}}{q} + \frac{2 \sqrt{nN(x+y_1)}}{q} \right) \ dx \ \ \textnormal{and}
\end{equation}

\begin{align} \label{oscillation of m}
    \frac{\partial^j }{\partial z^j}  \mathcal{I}( M_0 z,n;q) \ll_j  t^{(1+\epsilon) j}.
\end{align} Summarising the above simplifications, our main object of study now becomes
\begin{align}\label{supremum}
    S_{0,C}(N) \ll \sup_{u, \nu, y} \frac{N^{\frac{3}{2}} }{KQ^2}\sum_{q \sim C}\frac{1}{q}\sum_{d \mid q} d \sum_{m \leq M_0} \frac{\left|\lambda_g(m)\right| }{m^{1/4}} \left|\mathop{\sum\sum}_{\substack{ \ n \leq N_0\\m \equiv n (d)}} \frac{\lambda_f(n) }{n^{1/4}}  \mathcal{I}_{y, \nu}(m,n;q)\right|. 
\end{align}

\subsection{Size of the integral $\mathcal{I} (m,n;q)$} Let $K < t^{1 - \delta}$ for some positive $\delta$. We make a change of variable $\sqrt{x} = y$ in the equation \eqref{second last integral} so that the phase function essentially reduces to  
\begin{align*}
&     P(y)  = - \frac{t}{\pi} \log y + \frac{2 \sqrt{mN}}{q} y \pm  \frac{2 \sqrt{nN}}{q} \sqrt{y^2 +y_1} \\
    & P^\prime (y) = - \frac{t}{\pi  y} + \frac{2 \sqrt{mN}}{q} \pm  \frac{2 \sqrt{nN}}{q} \frac{y}{\sqrt{y^2 +y_1}  } \\
    & P^{\prime \prime}  (y) = \frac{t}{\pi  y^2} \pm \frac{2 \sqrt{nN}}{q} \frac{y_1}{(y^2 +y_1)^{3/2} } . 
\end{align*} Since $n \ll q^2 K^2/ N$ and $K < t^{1 - \delta}$, we notice that the second term is smaller that the first term and we obtain $|P^{\prime \prime}  (y)| \gg t$. Applying 
the Lemma \ref{sdb} we obtain that 
\begin{align}\label{bound for integral}
    \mathcal{I} (m,n;q) \ll t^{- \frac{1}{2}+\epsilon}
\end{align}

Estimating at this stage using equation \eqref{bound for integral} with $q \leq Q$ $N \sim t^2$ we obtain
\begin{align*}
S_{0,C}(N) & \ll \sup_{u, \nu, y} \frac{N^{\frac{3}{2}} }{KQ^2}\sum_{q \sim C}\frac{1}{q} \sum_{d \mid q} d \sum_{m \leq M_0} \frac{\left|\lambda_g(m)\right| }{m^{1/4}} \mathop{\sum}_{\substack{ \ n \leq N_0\\m \equiv n (d)}} \frac{ | \lambda_f(n) | }{n^{1/4}}   \left| \mathcal{I}_{y, \nu}(m,n;q)\right| \\ 
& \ll N^{\frac{1}{2}} \sum_{q \sim C}\frac{1}{q} \sum_{d \mid q} d  \mathop{\sum \sum}_{\substack{ m \leq M_0, \, n\leq N_0 \\ m \equiv n (d)}} t^{- \frac{1}{2}+ \epsilon}    \ll  N^{\frac{1}{2}}  t^{1+ \epsilon}. 
\end{align*} We are on boundary. To obtain additional saving we now apply Cauchy-Schwartz inequality and Poisson summation formula.

\subsection{Application of Cauchy's inequality and Poisson summation formula}
By Cauchy's inequality from \eqref{supremum} we infer that 
\begin{align}\label{Supremum}
      S_{0,C}(N) \ll \sup_{u, \nu, y} \frac{N^{\frac{3}{2}} M_0^{1/4} }{KQ^2}\sum_{q \sim C}\frac{1}{q}\sum_{d \mid q} d\sqrt{\mathcal{S}(M_0, N_0, d)}, \, \, \, \textrm{where, }
\end{align} 
 \[\mathcal{S}(M_0, N_0, d)=  \sum_{m \leq M_0}  \bigg|\sum_{\substack{n \leq N_0 \\ n \equiv m(d)}} \frac{\lambda_f(n) }{n^{1/4} }  \mathcal{I}_{y, \nu}(m,n;q)\bigg|^2.\] 
Now opening the absolute value square and smoothing out the $m$-sum we have 
\[ \mathcal{S}(M_0, N_0, d) \ll  \mathcal{O}, \ \ \ \textnormal{where}\]
\begin{align}\label{O}
    \mathcal{O}=\mathop{\sum \sum}_{\substack{n_1, n_2 \leq N_0 \\ n_1 \equiv n_2(d)}}\frac{\lambda_f(n_1)\lambda_f(n_2)}{(n_1 n_2)^{1/4}}
    \sum_{\substack{m \leq M_0 \\m \equiv n_1(d)}}\omega \left( \frac{m}{M_0}\right)\mathcal{I}_{y, \nu}(m,n_1;q)\overline{\mathcal{I}_{y, \nu}(m,n_2;q)}.
\end{align} Trivial estimate for $  \mathcal{O}$ gives
\begin{align*}
 \mathcal{O} \ll \frac{N_0^2 M_0}{d}  \times \frac{1}{t}. 
\end{align*}
%\subsection{Estimation of diagonal terms}
%\[\mathcal{D}=\sum_{n \leq N_0}\frac{|\lambda_f(n)|^2 }{\sqrt{n}}\sum_{\substack{m \leq M_0 \\ m \equiv n(d)}}  \omega \left( \frac{m}{M_0}\right) |\mathcal{I}_{y, \nu}(m,n;q)|^2 \]
%Poisson summation formula ... 
%\[\mathcal{D}=\frac{M_0}{d}\sum_{n \leq N_0}\frac{|\lambda_f(n)|^2 }{\sqrt{n}}\sum_{k \in \mathbb{Z}} e\left( \frac{-kn}{d}\right) \mathcal{I}(k, n, q),\]
%where
%\[  \mathcal{I}(k, n, q)=\int_{\mathbb{R}}\omega(z) e\left( \frac{kM_0z}{d}\right)|\mathcal{I}_{y, \nu}(M_0z,n;q)|^2 \ dz \]

%The Deligne's bound $\lambda_f(n) \ll \tau(n)$ yields the following trivial estimate for $\mathcal{D}$:
%\begin{align}\label{estimate D}
 %    \mathcal{D} \ll  M_0 \sqrt{N_0}/dt
%\end{align}

%\subsection{Estimation of off-diagonal terms}
We apply the Poisson summation formula to the  $m$-sum that appeared in \eqref{O}. We have the following

\begin{align*}
 S_1&:= \sum_{\substack{m \leq M_0 \\m \equiv n_1(d)}}\omega \left( \frac{m}{M_0}\right)\mathcal{I}_{y, \nu}(m,n_1;q)\overline{\mathcal{I}_{y, \nu}(m,n_2;q)} \\ 
 &= \frac{1}{d} \sum_{\alpha (d)} e \left( \frac{n_1 \alpha }{d} \right) \sum_m e \left( \frac{-m \alpha }{d} \right) \omega \left( \frac{m}{M_0}\right)\mathcal{I}_{y, \nu}(m,n_1;q)\overline{\mathcal{I}_{y, \nu}(m,n_2;q)}. 
\end{align*} By writing $m = \alpha_1 + \ell d$ and then applying Poisson summation formula with respect to $\ell$ we obtain
\begin{align*}
S_1 = \frac{1}{d} \sum_{\alpha (d)} e \left( \frac{n_1 \alpha }{d} \right)   \frac{M_0}{d} \sum_k \sum_{\alpha (d)} e \left( \frac{\alpha_1 ( \alpha +k )  }{d} \right)  \mathcal{I}(k, n_1, n_2,q) \, \, \, \, \, \, \, \,  \textrm{where, }
\end{align*}

\begin{align}
    \mathcal{I}(k, n_1, n_2,q)&=\int_{\mathbb{R}}\omega(z) e\left( \frac{kM_0z}{d}\right)\mathcal{I}_{y, \nu}(M_0z,n_1;q)\overline{\mathcal{I}_{y, \nu}(M_0z,n_2;q)} \ dz \notag\\
    &=\int_{\mathbb{R}}\omega(z) e\left( \frac{kM_0z}{d}\right) \ \int \int \left(\frac{x_2}{x_1}\right)^{it}\omega_{y,\nu}(x_1)\overline{\omega_{y,\nu}(x_2)} \times \notag\\
    &e\left( \frac{2 \sqrt{M_0zN}(\sqrt{x_1}-\sqrt{x_2})}{q} + \frac{2 \sqrt{N}(\sqrt{n_1 x_1+y}-\sqrt{n_2 x_2+y})}{q}\right) \ dx_1 \ dx_2  dz .
\end{align} Integrating by parts with respect to variable $z$, we observe that 
\begin{align*}
\mathcal{I}(k, n_1, n_2,q) \ll_j \left( 1 + t\right)^j \left( \frac{d}{ M_0 k }\right)^j. 
\end{align*} We observe  that the $z$-integral is negligibly small unless $k \ll \frac{td}{M_0}$. We obtain the following expression for $ \mathcal{O} $
\begin{align} \label{final O}
     \mathcal{O} = \frac{M_0}{d}\mathop{\sum \sum}_{\substack{n_1, n_2 \leq N_0 \\ n_1 \equiv n_2(d)}}\frac{\lambda_f(n_1)\lambda_f(n_2)}{(n_1 n_2)^{1/4}}\sum_{k \ll \frac{td}{M_0} } e\left( \frac{-kn_1}{d}\right) \mathcal{I}(k, n_1, n_2,q). 
\end{align} Let  $ \mathcal{O}_0 $ denotes the contributions of zero frequency ($k=0$) and $ \mathcal{O}_1 $ denotes the contribution of non-zero frequencies ( $k \neq 0$)

\subsection{Estimation of $\mathcal{O}_0$} By a change of variable $x_2=u x_1$ $x_1 = x$ we have
\begin{align*}
    \mathcal{I}(0, n_1, n_2,q)& =\int_{\mathbb{R}}\omega(z)  \ dz  \int \int u^{it}x\omega_{y,\nu}(x)\overline{\omega_{y,\nu}(ux)} \times \\
    &e\left( \frac{2 \sqrt{M_0 N zx}(1-u)}{q(1+\sqrt{u)}} +\frac{2\sqrt{N}(\sqrt{n_1}-\sqrt{n_2})\sqrt{x+y}}{q}\right. \\
    &\left. +\frac{2\sqrt{N}\sqrt{n_2}(1-u)x}{q(\sqrt{x+y}+\sqrt{ux+y})}\right)\ dx \ du  \ dz 
\end{align*}
Since $\max \{\sqrt{M_0N}/q, \sqrt{N_0N}/q\}\gg t^{-\epsilon}\sqrt{M_0N}/q$, the $u$ integral is negligible unless
\begin{equation}
\frac{\sqrt{M_0N}}{q}\gg t^{1-\epsilon}
\end{equation}
The $z$ integral is negligible unless
\begin{equation}\label{5.14}
\frac{\sqrt{M_0N}}{q}(1-u)\ll t^{\epsilon}
\end{equation}
Noting the inequality \eqref{5.14} we see that the $x$ integral is negligible unless 
\begin{equation}
\frac{\sqrt{N}(\sqrt{n_1}-\sqrt{n_2})}{q}=\frac{\sqrt{N}(n_1-n_2)}{(\sqrt{n_1}+\sqrt{n_2})q}\ll t^{\epsilon}
\end{equation}that is,
\begin{equation}
n_1-n_2\ll\frac{ t^{\epsilon}q\sqrt{N_0}}{\sqrt{N}}\ll t^{\epsilon}
\end{equation}So,  $\mathcal{I}(0, n_1, n_2,q)\ll t^{\epsilon-1}$ if $n_1-n_2\ll t^{\epsilon}$ and is negligible if $n_1-n_2\gg t^{\epsilon}$.

\subsection{Estimation of $ \mathcal{O}_1 $ }

We now analyse the integral  $ \mathcal{I}(k, n_1, n_2,q)$ for $k \neq 0$.  By a change of variable $x_2=u x_1$ $x_1 = x$ we obtain
\begin{align*}
    \mathcal{I}(k, n_1, n_2,q)& =\int_{\mathbb{R}}\omega(z) e\left( \frac{kM_0z}{d}\right) \ dz  \int \int u^{it}x\omega_{y,\nu}(x)\overline{\omega_{y,\nu}(ux)} \times \\
    &e\left( \frac{2 \sqrt{M_0 N zx}(1-\sqrt{u})}{q} + \frac{2 \sqrt{N}(\sqrt{n_1(x+y)}-\sqrt{n_2(ux+y)})}{q}\right) \ dx \ du  \ dz 
\end{align*}
As earlier, the $u$ integral is negligible unless
\begin{equation}\label{5.19}
\frac{\sqrt{M_0N}}{q}\gg t^{1-\epsilon}
\end{equation}
The $x$ integral is negligible unless
\begin{equation}\label{5.20}
\frac{\sqrt{M_0N}(u-1)}{q}\ll t^{-\epsilon}K 
\end{equation}
The $z$ integral is negligible unless $u>1$ and
\begin{equation*}
\frac{\sqrt{M_0N}(u-1)}{q}\asymp \frac{kM_0}{d} 
\end{equation*}
Denote $c := \frac{kM_0}{d}\left(\frac{\sqrt{M_0N}}{q}\right)^{-1} $. From the \eqref{5.19} and \eqref{5.20} we see that the $z$ integral is negligible unless
\begin{equation}\label{5.21}
u-1\ll c\ll K/t
\end{equation}Changing variable $u-1=cu'$ and multiplying a smooth cut-off function to restrict $u'\ll 1 $ we see that $\mathcal{I}$ becomes
\begin{equation*}
 \mathcal{I}(k, n_1, n_2,q)=c \int \hbox{( smooth weight functions)}\,\,\, e\left(\phi(u',x,y,z)\right) dx \ dy \ dz \ du 
\end{equation*} where
\begin{align*}
\phi(u',x,y,z)=&\frac{kM_0z^2}{d}-\frac{kM_0u'zx}{d}+\frac{t\log (1+cu')}{2\pi}+\frac{2\sqrt{N}(\sqrt{n_1}-\sqrt{n_2})\sqrt{x^2+y}}{q}\\
&+\hbox{lower order terms}
\end{align*} Now the $z$ integral is neglible unless 
\begin{equation}
z-\frac{u'x}{2}\ll \left(\frac{kM_0}{d}\right)^{-1}
\end{equation}Again changing variable $z-u'x/2=\left(\frac{kM_0}{d}\right)^{-1}z'$ and multiplying a new cut-off function for $z'\ll 1$, the phase function changes to
\begin{align*}
\phi(u',x,y,z)=&\left(\frac{kM_0}{4d}\right)u'^2x^2-\left(\frac{kM_0}{d}\right)^{-1}z'^2+\frac{t\log (1+cu')}{2\pi}+\frac{2\sqrt{N}(\sqrt{n_1}-\sqrt{n_2})\sqrt{x^2+y}}{q}\\
&+\hbox{lower order terms}
\end{align*} Now expanding the $\log(1+cu')$ (keeping the size \eqref{5.21} in mind) and using the second derivative bound for the $u'$ integral we get
\begin{equation}
 \mathcal{I}(k, n_1, n_2,q)\ll \frac{c\left(\frac{kM_0}{d}\right)^{-1}}{\sqrt{\frac{kM_0}{d}}}\ll t^{\epsilon-1}\sqrt{\frac{d}{kM_0}}
\end{equation}
\subsection{Applying the bounds} 
Combining the above estimates of $\mathcal{I}(k, n_1, n_2,q)$ we conclude 
\begin{align}\label{estimate O}
    \mathcal{O}&\ll_\epsilon \frac{M_0t^{\epsilon-1}}{d}\mathop{\sum \sum}_{\substack{n_1, n_2 \leq N_0 \\ n_1 \equiv n_2(d) }}{(n_1 n_2)^{\epsilon -1/4}}\sum_{0 < |k| \ll \frac{td}{M_0}} \sqrt{\frac{d}{M_0k}} +\frac{M_0t^{\epsilon-1}}{d}\mathop{\sum \sum}_{\substack{n_1, n_2 \leq N_0 \\ n_1 -n_2\ll t^{\epsilon}}}{(n_1 n_2)^{\epsilon -1/4}}\nonumber  \\
    & \ll N_0^{3/2}d^{-1} t^{\epsilon-1/2}+M_0 N_0^{1/2}d^{-1} t^{\epsilon-1}
    \end{align}
Hence combining the above two cases we obtain the following estimate for $S_{0,C}(N)$. 
\begin{align}\label{final}
      S_{0,C}(N) &\ll  \frac{N^{\frac{3}{2}}M_0^{\frac{1}{4}}}{KQ^2}\sum_{q \sim C}\frac{1}{q}\sum_{d \mid q} d\sqrt{ N_0^{3/2}d^{-1} t^{\epsilon-1/2}+M_0 N_0^{1/2}d^{-1} t^{\epsilon-1}} \nonumber\\
&\ll\frac{t^{\epsilon}N^{\frac{3}{2}}M_0^{\frac{1}{4}}}{KQ^2}\sum_{q\sim C}\frac{1}{q} \left( q^{1/2}N_0^{3/4}t^{-1/4}+q^{1/2}M_0^{1/2}N_0^{1/4}t^{-1/2}\right) \nonumber 
\end{align}Substituting the values $N_0=q^2K^2/N+K$, $M_0=q^2t^2/N+K$ (we consider only the first terms since the second terms contributes less)  and executing the $q$ sum we see that $S_{0,C}(N)$ is bounded by
\begin{equation}
S_{0,C}(N)\ll \frac{N^{\frac{3}{2}}}{KQ^2}\left(\frac{C^{5/2}K^{3/2}t^{1/4}}{N}+\frac{C^{5/2}K^{1/2}t}{N} \right)
\end{equation}Substituting this in \eqref{dyadic} we get
\begin{equation}
\frac{S(N)}{\sqrt{N}}\ll N^{1/4}t^{1/4}K^{1/4} + \frac{N^{1/4}t}{K^{3/4}}
\end{equation} With the optimal choice $K=t^{3/4}$ we get that
\begin{equation}
\frac{S(N)}{\sqrt{N}}\ll_{\epsilon} t^{1-1/16+\epsilon }
\end{equation}and Theorem \ref{main} follows.

 {\bf Acknowledgement:} The authors would like to thank  Prof. Ritabrata Munshi for helpful suggestions and comments. They also thank Indian Statistical Institute Kolkata for wonderful academic atmosphere. During the work,  R. Acharya, who is a NBHM Post-doctoral fellow at RKMVERI, Belur Math, was supported by the Department of Atomic Energy, Government of India (DAE file reference  no. 0204/12/2018/R \& D II/ 6463). For this work, S. Singh was partially  supported by D.S.T. inspire faculty fellowship no.   DST/INSPIRE/$04/2018/000945$. 

{}

\end{document}